\documentclass[a4paper,leqno,10pt]{amsart}

\usepackage{epsf,graphicx,epsfig}
\usepackage{amscd}
\usepackage{amsmath,latexsym,amssymb,amsthm}
\usepackage[nospace,noadjust]{cite}
\usepackage{textcomp,mathtools}
\usepackage{setspace,cite}
\usepackage{lscape,fancyhdr,fancybox}
\usepackage{stmaryrd}
\usepackage[all,cmtip]{xy}
\usepackage{tikz}
\usepackage{cancel}
\usetikzlibrary{shapes,arrows,decorations.markings}
\usepackage{graphicx}

\usepackage{color}
\usepackage{url}
\usepackage{enumerate}
\usepackage[mathscr]{euscript}

  \theoremstyle{plain}

\swapnumbers
    \newtheorem{thm}{Theorem}[section]

   \newtheorem{lemma}[thm]{Lemma}

    \newtheorem{subsec}[thm]{}
\theoremstyle{definition}
    \newtheorem{defn}[thm]{Definition}
        \newtheorem{remark}[thm]{Remark}
\theoremstyle{remark}

\newcommand{\threepartdef}[6]
{
	\left\{
		\begin{array}{lll}
			#1 & \mbox{if } #2 \\
			#3 & \mbox{if } #4 \\
			#5 & \mbox{if } #6
		\end{array}
	\right.
}

\title{}
\author{}
\date{}
\usepackage{amssymb}

\usepackage{hyperref}
\hypersetup{
	colorlinks,
	citecolor=blue,
	filecolor=black,
	linkcolor=blue,
	urlcolor=black
}

\begin{document}
\title[Calculus Structure on Secondary Hochschild (co)homology]{Noncommutative Differential Calculus Structure on Secondary Hochschild (co)homology}

\author{Apurba Das}
\address{Department of Mathematics and Statistics,
Indian Institute of Technology, Kanpur 208016, Uttar Pradesh, India.}
\email{apurbadas348@gmail.com}

\author{Satyendra Kumar Mishra}\footnote{Corresponding Author [A2] email: satyamsr10@gmail.com.}
\author{Anita Naolekar}
\address{Statistics and Mathematics Unit,
Indian Statistical Institute, Bangalore- 560059, Karnataka, India.}
\email{satyamsr10@gmail.com, anita@isibang.ac.in}

\curraddr{}
\email{}


\begin{abstract}
Let $B$ be a commutative algebra and $A$ be a $B$-algebra (determined by an algebra homomorphism $\varepsilon:B\rightarrow A$). M. D. Staic introduced a Hochschild like cohomology $H^{\bullet}((A,B,\varepsilon);A)$ called secondary Hochschild cohomology, to describe the non-trivial $B$-algebra deformations of $A$. J. Laubacher et al later obtained a natural construction of a new chain (and cochain) complex $\overline{C}_{\bullet}(A,B,\varepsilon)$ (resp. $\overline{C}^{\bullet}(A,B,\varepsilon)$) in the process of introducing the secondary cyclic (co)homology. It turns out that unlike the classical case of associative algebras (over a field), there exist different (co)chain complexes for the $B$-algebra $A$. In this paper, we establish a connection between the two (co)homology theories for $B$-algebra $A$. We show that the pair $\big(H^{\bullet}((A,B,\varepsilon);A),HH_{\bullet}(A,B,\varepsilon)\big)$ forms a noncommutative differential calculus, where $HH_{\bullet}(A,B,\varepsilon)$ denotes the homology of the complex $\overline{C}_{\bullet}(A,B,\varepsilon)$.

\end{abstract}

\noindent

\thispagestyle{empty}

\maketitle


\vspace{0.2cm}

\section{Introduction}
Hochschild cohomology was introduced by G. Hochschild to study certain extensions of associative algebras. In 1963, a pioneer work of M. Gerstenhaber related Hochschild cohomology with the deformations of associative algebras \cite{gers}. He further proved that the Hochschild cohomology of an associative algebra carries a rich algebraic structure which is now known as {\it Gerstenhaber algebra} (see \cite{gers} for details). These structures appear in the context of the exterior algebra of Lie algebras, multivector fields on smooth manifolds, and differential forms on Poisson manifolds.  A Gerstenhaber algebra is a graded commutative associative algebra $(\mathcal{A} = \oplus_{i \in \mathbb{Z}} \mathcal{A}^i, \cup)$ together with a degree $-1$ graded Lie bracket $[-,-]$ on $\mathcal{A}$ satisfying the following compatibility condition.
\begin{align*}
[a, b \smile c] = [a, b] \smile c + (-1)^{(|a|-1) |b|} b \smile [a,c].
\end{align*}
A Batalin-Vilkovisky (BV-) operator on a Gerstenhaber algebra $(\mathcal{A} = \oplus_{i \in \mathbb{Z}} \mathcal{A}^i, \smile)$ is a square zero degree $-1$ map $\triangle : \mathcal{A} \rightarrow \mathcal{A}$ with
\begin{align*}
[a,b] = \pm ( \triangle (a \smile b) - (\triangle a) \smile b - (-1)^{|a|} a \smile (\triangle b)).
\end{align*}
 This shows that the bracket $[-,-]$ obstructs $\triangle$ to be a derivation with respect to the product $\smile$. A Gerstenhaber algebra with a BV-operator is called a Batalin-Vilkovisky algebra (BV-algebra). 
 

Let $A$ is an associative $k$-algebra, $B$ a commutative $k$-algebra and $\varepsilon : B \rightarrow A$ an algebra morphism with $\varepsilon (B) \subset \mathcal{Z}(A)$, the center of $A$. In 2016, M. Staic  \cite{Staic1,Staic2} introduced the secondary Hochschild cohomology $H^\bullet ((A, B, \varepsilon); M)$ of the triple $(A, B, \varepsilon)$ with coefficients in a $B$-symmetric $A$-bimodule $M$. This cohomology, motivated by an algebraic version of the
second Postnikov invariant \cite{Staic1}, controls the deformations of the $B$-algebra structures on $A[[t]]$. When $B$ is the underlying field $k$, the secondary Hochschild cohomology coincides with the classical Hochschild cohomology. Several results which are true for classical Hochschild cohomology theory for an associative algebra with coefficients in itself have analogues for the secondary Hochschild cohomology $H^\bullet ((A, B, \varepsilon), A)$ of a triple $(A, B, \varepsilon)$ with coefficients in $A$.  In particular, it is proved in \cite{Staic3} that the secondary complex $C^{\bullet}((A,B,\varepsilon);A )$ is a multiplicative non-symmetric operad, which induces a natural homotopy Gerstenhaber algebra structure on it. Consequently, one obtains a natural Gerstenhaber algebra structure on the secondary Hochschild cohomology $H^{\bullet}((A,B,\varepsilon);A )$.




The simplicial structure of the complex $C^{\bullet}((A,B,\varepsilon);M)$ that defines the secondary cohomology $H^\bullet ((A, B, \varepsilon); M)$ is discussed in \cite{Staic4}. For this purpose, the authors introduce a simplicial module $\mathcal{B}(A, B, \varepsilon)$ over a certain simplicial algebra, and call it the {\it secondary bar simplicial module}. This simplicial module $\mathcal{B}(A, B, \varepsilon)$ is the analogue of the bar resolution associated with a k-algebra A, in Hochschild cohomology.
Using this simplicial module, the cohomology $H^{\bullet}((A,B,\varepsilon);M )$ of the triple $(A, B, \varepsilon)$ with coefficients in $M$ can be realized as the homology of the associated complex of co-simplicial module. This construction of the secondary bar simplicial module leads to natural constructions of the secondary Hochschild homology groups $H_\bullet((A, B, \varepsilon), M)$ of the triple $(A, B, \varepsilon)$ with coefficients in $M$. Also the authors in \cite{Staic4} have given natural constructions of secondary cohomology groups $HH^{\bullet}(A,B,\varepsilon)$ and the  secondary cyclic cohomology groups $HC^{\bullet}(A,B,\varepsilon)$ associated to a triple $(A,B,\varepsilon)$. This has also prompted the authors to define the corresponding homology groups $HH_{\bullet}(A,B,\varepsilon)$ and $HC_{\bullet}(A,B,\varepsilon)$ associated to the triple $(A, B, \varepsilon)$. We refer to \cite{Staic4} and \cite{BV-2020} for more Hochschild-like results on the secondary Hochschild and secondary cyclic (co)homology groups, associated to the triple $(A, B, \varepsilon)$.


This article aims to establish a connection between the two (co)homology theories for a triple $(A,B,\varepsilon)$ introduced in \cite{Staic2} and \cite{Staic4}. We show that the pair $\big(H^{\bullet}((A,B,\varepsilon);A),HH_{\bullet}(A,B,\varepsilon)\big)$ forms a noncommutative differential calculus. A calculus $(\mathcal{A},\Omega)$ consists of a Gerstenhaber algebra $\mathcal{A}$ and a graded space $\Omega$ such that 
\begin{itemize}
\item $\Omega$ carries a Gerstenhaber module structure over $\mathcal{A}$, and
\item there exists a differential $B:\Omega_{\bullet}\rightarrow \Omega_{\bullet+1}$ satisfying the Cartan-Rinehart homotopy formula (see Definition \ref{BV module}).
\end{itemize} 
In differential geometry and noncommutative geometry, there are several concrete examples of noncommutative  differential calculi \cite{gelfand-daletskii-tsygan,Nest, Tamarkin}. These examples include the classical calculus of multivector fields and differential forms, a calculus on the pair Hochschild cohomology and  Hochschild homology of associative algebras, the calculus for Poisson structures, and the calculus for general left Hopf algebroids with respect to general coefficients (see Section $6$, \cite{Kowalzig}). In \cite{Kowalzig} N. Kowalzig proved that if a comp module structure \cite{KoKr} is cyclic over a multiplicative operad, then this cyclic comp module structure induces a noncommutative differential calculus (in the sense of \cite{Nest, Tamarkin}) on the pair of the associated homology of the cyclic $k$-module and the cohomology of the operad. 

In this article, we consider the multiplicative non-symmetric operad structure on secondary cochain complex $C^{\bullet}((A,B,\varepsilon);A)$ from \cite{Staic3}. The secondary Hochschild chain complex associated to a triple $(A,B,\varepsilon)$ is denoted by $\overline{C}_{\bullet}(A,B,\varepsilon)$. We define comp module actions of the operad $C^{\bullet}((A,B,\varepsilon);A)$ on the complex $\overline{C}_{\bullet}(A,B,\varepsilon)$. With these actions, we prove that the complex $\overline{C}_{\bullet}(A,B,\varepsilon)$ is a cyclic comp module over the operad $C^{\bullet}((A,B,\varepsilon);A)$. Subsequently, it follows that the pair of underlying homologies $\big(H^{\bullet}((A,B,\varepsilon);A),HH_{\bullet}(A,B,\varepsilon)\big)$ forms a (noncommutative) differential calculus.

In Section $2$, we recall the definitions and results related to multiplicative non-symmetric operads and cyclic unital comp module structures over multiplicative operads. In Section $3$, we recall secondary Hochschild (co)homology for a triple $(A,B,\varepsilon)$. In particular, we recall the definitions of the complexes $C^{\bullet}((A,B,\varepsilon);A)$ and $\overline{C}_{\bullet}(A,B,\varepsilon)$ from \cite{Staic2, Staic4}. We also recall the multiplicative operad structure on the complex $C^{\bullet}((A,B,\varepsilon);A)$ from \cite{Staic3}. In Section $4$, we define comp module actions of $C^{\bullet}((A,B,\varepsilon);A)$ on the complex $\overline{C}_{\bullet}(A,B,\varepsilon)$ and obtain a comp module structure on $\overline{C}_{\bullet}(A,B,\varepsilon)$. We further consider an operator $\mathfrak{t}:\overline{C}_{\bullet}(A,B,\varepsilon)\rightarrow \overline{C}_{\bullet}(A,B,\varepsilon)$ to show that the comp module structure is cyclic. Finally, we conclude that there exists a (noncommutative) differential calculus on the pair $\big(H^{\bullet}((A,B,\varepsilon);A),HH_{\bullet}(A,B,\varepsilon)\big)$. This provides a connection between the cohomology groups of the triple $(A, B, \varepsilon)$ and the homology groups associated to the triple $(A, B, \varepsilon)$.

\section{Comp modules over (non-symmetric) operads}

In this section, we recall the notion of cyclic comp modules over operads and the related results from \cite{Kowalzig}. In particular, we recall that a cyclic comp module structure over a multiplicative operad induces a noncommutative differential calculus. 
\begin{defn}
A non-symmetric (unital) operad $\mathcal{O}$ in the category of $k$-modules is a sequence of $k$-modules $\{\mathcal{O}^{n}\}_{n\geq 0}$ equipped with $k$-linear maps
$$\circ_i:\mathcal{O}^n\otimes \mathcal{O}^m\rightarrow \mathcal{O}^{n+m-1},\quad \mbox{for}~~ m,n\geq 0 \mbox{~~and } 0\leq i\leq n,$$ 
and an element $\textbf{1}\in \mathcal{O}^1$ such that the following identities hold
$$f\circ_i g=0 \quad \quad\mbox{if } n<i \quad\mbox{ or \quad}  n=0,$$
$$(f\circ_i g)\circ_j h = \threepartdef
{(f\circ_j h)\circ_{i+p-1} g}      {j<i,}
{f\circ_i (g\circ_{j-i+1} h)}      {i\leq j< m+i,}
{(f\circ_{j-m+1} h)\circ_i f} {j\geq m+i,}$$
and $$\quad f\circ_i \textbf{1}=f =\textbf{1}\circ_1 f,\quad \quad\mbox{\quad \quad \quad for all  } i\leq n,$$
where $f\in \mathcal{O}^n$, $g\in \mathcal{O}^m$, and $h\in \mathcal{O}^p$. 
\end{defn}
If $\mathcal{O}$ is an operad, then one can define a circle product $\circ: \mathcal{O}^n \otimes \mathcal{O}^{m}\rightarrow \mathcal{O}^{n+m-1}$ by $$f\circ g = \sum_{i=1}^n (-1)^{(i-1)(m-1)}f\circ_i g,\quad \mbox{for~~} f\in \mathcal{O}^{n} \mbox{~~and~~} g\in \mathcal{O}^{m}.$$ 
Subsequently, a degree $-1$ bracket is given by
$$[f,g]=f\circ g-(-1)^{(n-1)(m-1)}g\circ f,\quad \mbox{for~~} f\in \mathcal{O}^{n} \mbox{~~and~~} g\in \mathcal{O}^{m}.$$ 

We call the operad $\mathcal{O}$, a multiplicative operad if there exists an element $\mu\in \mathcal{O}^2$ and an element $e\in \mathcal{O}^0$ such that $\mu\circ_1 \mu=\mu\circ_2\mu$ and $\mu\circ_1 e=\textbf{1}=\mu\circ_2 e$. The multiplication $\mu$ on the operad $\mathcal{O}$ induces a differential $\delta_\mu:\mathcal{O}^n\rightarrow \mathcal{O}^{n+1},$ given by $\delta_\mu(f)=[\mu,f]$ for $f\in \mathcal{O}^n$. Let us denote by $H^{\bullet}_{\mu}(\mathcal{O})$, the cohomology space of the complex $(\mathcal{O},\delta_\mu)$. The multiplication $\mu$ also induces a cup product $\smile: \mathcal{O}^n \otimes \mathcal{O}^{m}\rightarrow \mathcal{O}^{n+m}$ on the complex $\mathcal{O}$, which is given by 
$$f\smile g = (\mu\circ_2 f)\circ_1 g, \quad \mbox{for~~} f\in \mathcal{O}^{m} \mbox{~~and~~} g\in \mathcal{O}^{n}.$$
The graded commutative product $\smile$ and the Gerstenhaber bracket $[~,~]$ induce a Gerstenhaber algebra structure on the cohomology $H^{\bullet}_{\mu}(\mathcal{O})$ (see \cite{gers-voro} for more details).

\begin{defn}
A unital (left) comp module $\mathcal{M}_{\tiny{\bullet}}$ over an operad $\mathcal{O}^{\bullet}:=\{\mathcal{O}^n\}_{n\geq 0}$ is a sequence of $k$-modules $\{\mathcal{M}_n\}_{n\geq 0}$ equipped with comp module maps
$\bullet_i:\mathcal{O}^m\otimes \mathcal{M}_p\rightarrow \mathcal{M}_{p-m+1},$
for $1\leq i\leq p-m+1$ and $0\leq m\leq p$ such that 
\begin{equation}\label{con1}
f\bullet_i (g\bullet_j x) = \threepartdef
{g\bullet_j (f\bullet_{i+n-1} x)}      {j<i,}
{(f\circ_{j-i+1} g)\bullet_i x}      {j-m<i\leq j,}
{g\bullet_{j-m+1} (f\bullet_i x)} {1\leq i\leq j-m,}
\end{equation} 
\begin{equation}\label{unital}
\textbf{1}\bullet_i x=x \quad\mbox { for } 1\leq i \leq p, 
\end{equation}
where $f\in \mathcal{O}^m,$ $g\in \mathcal{O}^n$, and $x\in \mathcal{M}_p,$ for $m\geq 0,~ n,p\geq 0$. The maps $\bullet_i$ are zero, for $m>p$.
\end{defn}

\begin{defn}
A unital comp module $\mathcal{M}_{\bullet}$ is called para-cyclic if it is additionally equipped with 
\begin{enumerate}[(i)]
\item an extra comp module maps $\bullet_0:\mathcal{O}^m\otimes \mathcal{M}_p\rightarrow \mathcal{M}_{p-m+1},$
for $0\leq m\leq p+1$ (the maps $\bullet_0$ are assumed to be zero, for $m>p+1$) such that the relations \eqref{con1}-\eqref{unital} also hold true for $i=0$, and

\item a $k$-linear map $\mathfrak{t}:\mathcal{M}_p\rightarrow \mathcal{M}_p,$ for $p\geq 1,$ such that 
$$\mathfrak{t}(f\bullet_i x)=f\bullet_{i+1}\mathfrak{t}(x),$$
for $f\in \mathcal{O}^m$, $x\in \mathcal{M}_p$, and $0\leq i\leq p-m$.
If the map $\mathfrak{t}:\mathcal{M}_p\rightarrow \mathcal{M}_p$ satisfies the condition: $\mathfrak{t}^{p+1}=Id$, then the para-cyclic comp module $\mathcal{M}_{\bullet}$ is said to be `\textbf{cyclic}' comp module over the operad $\mathcal{O}^{\bullet}$. 
 \end{enumerate}
\end{defn}

Let $(\mathcal{O}^{\bullet},\mu)$ be a multiplicative non-symmetric operad and $\mathcal{M}_{\bullet}$ be a cyclic comp module over $\mathcal{O}^{\bullet}$. Then let us recall from \cite{Kowalzig} that there is a cyclic $k$-module structure on $\mathcal{M}_{\bullet}$ with cyclic operator $\mathfrak{t}:\mathcal{M}_p\rightarrow \mathcal{M}_p$, face maps $d_i:\mathcal{M}_p\rightarrow \mathcal{M}_{p-1}$ and degeneracies $s_j:\mathcal{M}_p\rightarrow \mathcal{M}_{p+1}$ defined as follows
\begin{align*}
d_i(x)=~&\mu \bullet_i x, \quad \quad\mbox{for}~i=0,\ldots,p-1,\\
d_p(x)=~&\mu \bullet_0 \mathfrak{t}(x),\\
s_j(x)=~&\textbf{1} \bullet_{j+1} x, \quad \quad\mbox{for}~j=0,\ldots,p,
\end{align*}
where $x\in \mathcal{M}_p$. Thus, a simplicial boundary $\mathfrak{b}$, a norm operator $N$, an extra degeneracy $s_{-1}$, and the cyclic differential $B$ are defined as follows 
\begin{equation*}
\mathfrak{b}:= \sum\limits_{i=0}^{p}(-1)^i d_i,\quad N:=\sum\limits_{i=0}^{p}(-1)^{ip} \mathfrak{t}^i,\quad s_{-1}:=\mathfrak{t}s_p,\quad B:=(id-\mathfrak{t})s_{-1}N.
\end{equation*}

The normalised complex $\mathfrak{N}(\mathcal{M})_{\bullet}$ is the quotient of the complex $\mathcal{M}_{\bullet}$ by the subcomplex spanned by the images of degeneracy maps $s_j,$ for $j=0,1,\ldots,p$. The map $B$ on the normalised complex $\mathfrak{N}(\mathcal{M})_{\bullet}$ is given by
\begin{equation}\label{Norm-Operator}
B(x)=s_{-1}N(x)=\sum_{i=0}^p (-1)^{ip}~ \textbf{1}\bullet_0 \mathfrak{t}^i (x),\quad \mbox{for  }~~x\in \mathfrak{N}(\mathcal{M})_{p}.
\end{equation}

\subsection{The noncommutative differential calculus associated to a cyclic comp module over a multiplicative operad}

\begin{defn}[\cite{Kowalzig}]\label{BV module}
A graded $k$-module $\Omega:=\oplus \Omega_n$ is called a Gerstenhaber module over Gerstenhaber algebra 
$\mathcal{A}$ if there exist maps $\mathfrak{i}:\mathcal{A}_m\otimes \Omega_p\rightarrow \Omega_{p-m}$, and $\mathcal{L}: \mathcal{A}_m\otimes \Omega_p\rightarrow \Omega_{p-m+1}$ such that 
\begin{enumerate}[(i)]
\item the action $\mathfrak{i}$ makes $\Omega$ a graded module over the graded commutative associative algebra $(\mathcal{A},\wedge)$;
\item the action $\mathcal{L}$ makes $\Omega$ a graded Lie module over the graded Lie algebra $(\mathcal{A}[1],[-,-])$;
\item for any $X\in \mathcal{A}_m$ and $Y\in \mathcal{A}_{n+1}$, the following relation holds
$$\mathfrak{i}_{[X,Y]}=\mathfrak{i}_{X}\mathcal{L}_Y-(-1)^{mn}\mathcal{L}_{Y}\mathfrak{i}_X .$$
\end{enumerate}
The Gerstenhaber module $\Omega=\oplus \Omega_n$ is called a Batalin-Vilkovisky module if there exists a $k$-linear map $B:\Omega_n\rightarrow \Omega_{n+1}$ such that $B^2=0$, and it satisfies the following Cartan-Rinehart homotopy formula
$$\mathcal{L}_{X}=B\circ\mathfrak{i}_X-(-1)^m \mathfrak{i}_{X}\circ B, \quad\mbox{for  ~}X\in \mathcal{A}_m.$$
A pair $(\mathcal{A},\Omega)$, where $\mathcal{A}$ is a Gerstenhaber algebra and $\Omega$ is a Batalin-Vilkovisky module over $\mathcal{A}$, is called a \textbf{noncommutative differential calculus}.
\end{defn}

Let $(\mathcal{O},\mu)$ be a multiplicative operad and $\mathcal{M}_{\bullet}$ be a cyclic unital comp module over $(\mathcal{O},\mu)$. 
 
\begin{defn}[\cite{Kowalzig}]
For ${f}\in \mathcal{O}^m$, the cap product and the Lie derivative are given as follows.

\begin{enumerate}[(i)]
\item The cap product $i_{f}: \mathcal{M}_p\rightarrow \mathcal{M}_{p-m}$ is defined by
$$i_{f}x= (\mu\circ_2 f)\bullet_0 x, \quad \quad\mbox{for~~} x\in \mathcal{M}_p.$$
\item The Lie derivative $\mathcal{L}_f :\mathcal{M}_p\rightarrow\mathcal{M}_{p-m+1}$ of $x\in \mathcal{M}_p$ along an element $f\in \mathcal{O}^m$ is defined by 
\begin{equation*}
L_{f}(x)= \threepartdef
{\sum\limits_{i=1}^{n-p+1}(-1)^{(p-1)(i-1)}{f}\bullet_i x+\sum\limits_{i=1}^{p}(-1)^{n(i-1)+p-1}{f}\bullet_0 \mathfrak{t}^{(i-1)}(x)}      {m<p+1,}
{(-1)^{p-1}{f}\bullet_0 N(x)}      {m=p+1,}
{0} {m>p+1.}
\end{equation*} 
\end{enumerate}
\end{defn}
It is shown in \cite{Kowalzig} that the cap product and the Lie derivative satisfy the following identities.
\begin{equation}\label{descend}
i_{\delta_\mu f}=\mathfrak{b} \circ i_{f }- (-1)^m i_{f } \circ \mathfrak{b},\quad\quad [\mathfrak{b}, \mathcal{L}_{f }]+ \mathcal{L}_{\delta_\mu {f}}=0.
\end{equation}
\begin{equation}\label{graded module}
i_{f}i_{g}=i_{f\smile g}, \quad\quad [\mathcal{L}_{f},\mathcal{L}_{g}]=\mathcal{L}_{[f, g]}.
\end{equation}
From equation \eqref{descend}, it follows that the cap product and Lie derivative descend to well-defined operators on the homology $H_{\bullet}(\mathcal{M}_{\bullet})$. Moreover, from equation \eqref{graded module}, the cap product makes $H_{\bullet}(\mathcal{M}_{\bullet})$ a graded module over the algebra $(H^{\bullet}(\mathcal{O}^{\bullet}),\smile)$, and the Lie derivative makes $H_{\bullet}(\mathcal{M}_{\bullet})$ a graded Lie module over the graded Lie algebra $(H^{\bullet+1}(\mathcal{O}^{\bullet}),[~,~])$. For any two cocycles $f\in \mathcal{O}^m$ and $g\in \mathcal{O}^n$, the induced operators $$\mathcal{L}_{f}:H_\bullet(\mathcal{M}_{\bullet})\rightarrow H_{\bullet-m+1}(\mathcal{M}_{\bullet})\quad \mbox{and}\quad i_g :H_\bullet(\mathcal{M}_{\bullet})\rightarrow H_{\bullet-n}(\mathcal{M}_{\bullet})$$ satisfy the relation
\begin{equation}\label{precalculus}
[i_{f},\mathcal{L}_g]=i_{[f,g]}.
\end{equation}
The identity \eqref{precalculus} shows that the homology $H_{\bullet}(\mathcal{M})$ is a Gerstenhaber module over the Gerstenhaber algebra $H_{\mu}^{\bullet}(\mathcal{O})$. 

In fact, this Gerstenhaber module structure extends to a Batalin-Vilkovisky module structure. Let us consider the normalised complexes $\mathfrak{N}(\mathcal{M})_{\bullet}$ and $\mathfrak{N}(\mathcal{O})^{\bullet}$. Also, recall that (co)homology of the normalised (co)chain complex is the same as the (co)homology of the original complex. The induced norm operator on the normalised complex $\mathfrak{N}(\mathcal{M})_{\bullet}$ induces a well defined $k$-linear map $B:H_{\bullet}(\mathcal{M})\rightarrow H_{\bullet+1}(\mathcal{M})$ satisfying $B^2=0$. For any $m$-cocycle $f\in\mathfrak{N}(\mathcal{O})^{m}$, the operators $$\mathcal{L}_{f}:H_\bullet(\mathcal{M})\rightarrow H_{\bullet-m+1}(\mathcal{M})\quad \mbox{and}\quad i_f :H_\bullet(\mathcal{M})\rightarrow H_{\bullet-m}(\mathcal{M})$$
 satisfy the following Cartan-Rinehart homotopy formula
$$\mathcal{L}_f=[B,i_f]:=B\circ i_f -(-1)^m i_f\circ B.$$
i.e., $H_{\bullet}(\mathcal{M})$ is a Batalin-Vilkovisky module over the Gerstenhaber algebra $H_{\mu}^{\bullet}(\mathcal{O})$. Thus, the pair $\big(H_{\mu}^{\bullet}(\mathcal{O}),H_{\bullet}(\mathcal{M})\big)$ forms a noncommutative differential calculus (see \cite{Kowalzig} for more details).


\section{Secondary Hochschild (co)homology}
Let $A$ be an associative $k$-algebra and $B$ be a commutative $k$-algebra. Suppose that there is a $k$-algebra morphism $\varepsilon : B \rightarrow A$ such that $\varepsilon (B)  \subset \mathcal{Z}(A)$ the center of $A$. We denote the above structure by a triple $(A, B, \varepsilon)$. In this section, we recall the notion of secondary Hochschild (co)homology for a triple $(A, B, \varepsilon)$ introduced in \cite{Staic2, Staic3, Staic4}. 

\subsection{Secondary Hochschild (co)homology of the triple $(A,B,\varepsilon)$ with coefficients in $A$} Let $(A, B, \varepsilon)$ be a triple. With the above notations, the secondary Hochschild cohomology of a triple $(A, B, \varepsilon)$ with coefficients in $A$ is given by considering the cochain complex $C^{\bullet}((A,B,\varepsilon);A)=\oplus_{n\geq 0}C^{n}((A,B, \varepsilon); A)$, where
\begin{equation*}
C^{n}((A,B, \varepsilon); A):=  Hom_k (A^{\otimes n} \otimes B^{\otimes \frac{n(n-1)}{2}} , A),\quad\mbox{for } n\geq 0,
\end{equation*}
and the differential $\delta^{\varepsilon}: C^{n-1}((A,B, \varepsilon); A) \rightarrow C^{n}((A,B, \varepsilon); A)$ is defined by

\begin{align*}
\delta^\varepsilon(f)\left(\otimes \left( {\begin{array}{ccccc}
a_1   & b_{1,2}   & \cdots   & b_{1,n-1}   & b_{1,n}\\
1       & a_2      & \cdots   & b_{2, n-1}  & b_{2,n} \\
\cdot  & \cdot    & \cdots   & \cdot     & \cdot\\
1       & 1         & \cdots   &  a_{n-1}    & b_{n-1, n}\\
1       & 1         & \cdots   & 1           & a_n
\end{array} } \right)\right):=a_1 \varepsilon (b_{1,2}\cdots b_{1, n}) f \left(\otimes \left( {\begin{array}{ccccc}
a_2    & b_{2,3}   & \cdots   & b_{2,n}\\
1       & a_{3}    & \cdots   & b_{3,n} \\
\cdot  & \cdot    & \cdots   & \cdot    \\
1       & 1         & \cdots         & a_n
\end{array} } \right)\right)+
\end{align*}
\begin{align*}
& \sum_{i=1}^{n-1}(-1)^i f\left(\otimes \left( {\begin{array}{ccccccc}
a_1  & b_{1,2}& \cdots &  b_{1, i} b_{1, i+1}  & \cdots  & b_{1,n-1}& b_{1,n} \\
1 & a_2 & \cdots & b_{2, i} b_{2, i+1} &  \cdots& b_{2,n-1}& b_{2, n} \\
\cdot& \cdot &\cdots &\cdot& \cdots &\cdot& \cdot\\
1 &  1 & \cdots & a_i a_{i+1}\varepsilon (b_{i, i+1}) & \cdots & b_{i, n-1} b_{i+1, n-1} & b_{i,n} b_{i+1,n} \\
\cdot & \cdot & \cdots & \cdot & \cdots & \cdot & \cdot \\
1 & 1 & \cdots &1 & \cdots &a_{n-1} & b_{n-1, n} \\
1 & 1 & \cdots & 1 & \cdots & 1 &  a_{n}
\end{array} } \right)\right)+\\
&(-1)^n f \left(\otimes \left( {\begin{array}{ccccc}
a_1   & b_{1,2}   & \cdots   & b_{1,n-1}\\
1       & a_{2}    & \cdots   & b_{2,n-1} \\
\cdot  & \cdot    & \cdots   & \cdot    \\
1      & 1      & \cdots  & a_{n-1}
\end{array}} \right)\right)\varepsilon (b_{1,n} \cdots b_{n-1, n}) a_{n}.
\end{align*}

The cohomology of this cochain complex is denoted by  $H^\bullet ((A, B , \varepsilon); A)$ and it is called the secondary Hochschild cohomology of the triple $(A, B, \varepsilon)$ with coefficients in $A$.

In the present article, we do not use the notion of secondary Hochschild homology $H_\bullet ((A, B, \varepsilon) ; A)$ of the triple $(A, B, \varepsilon)$ with coefficients in $A$. So, we skip the definition of $H_\bullet ((A, B, \varepsilon) ; A)$ (see \cite{Staic1} for details).

\begin{remark}
If $B = k$ and $\varepsilon : k \rightarrow A$ defining the $k$-algebra structure on $A$, the secondary Hochschild (co)homology coincides with the classical Hochschild (co)homology of the associative algebra $A$.
\end{remark}

\subsection{Gerstenhaber algebra structure on the secondary Hochschild cohomology}
Let us consider an operad $\{\mathcal{O}^n\}_{n\geq 0}$, where 
$$\mathcal{O}^n= C^{n}((A,B,\varepsilon);A).$$
The underlying $k$-bilinear maps 
\begin{align*}
\circ_i : C^n  ((A, B , \varepsilon); A) \otimes C^m  ((A, B , \varepsilon); A) \rightarrow C^{n+m-1}  ((A, B , \varepsilon); A),\quad \mbox{for}~~ n, m \geq 0 \mbox{  and }1 \leq i \leq n,
\end{align*}
are given as follows 
\begin{align}\label{comp products}
&(f\circ_i g)(T^1_{n+m-1})\\\nonumber
=& f\left(\otimes \left( {\begin{array}{ccccccccc}
a_1 &\cdots    & b_{1,i-1} & \prod\limits_{j= i}^{m+i-1}b_{1,j} & b_{1,m+i}&\cdots   & b_{1,n+m-1} \\
1   &\cdots    & b_{2,i-1} & \prod\limits_{j= i}^{m+i-1}b_{2,j} & b_{2,m+i}&  \cdots   & b_{2, n+m-1}\\
\cdot  & \cdot &\cdot   & \cdot & \cdot &\cdot  & \cdot   \\
1 &\cdots    & a_{i-1} & \prod\limits_{j= i}^{m+i-1}b_{i-1,j} & b_{i-1,m+i}&\cdots   & b_{i-1,n+m-1} \\
1 &\cdots    & 1 & g(T^{i}_{m+i-1})&\prod\limits_{j= i}^{m+i-1}b_{j,m+i}&\cdots   & \prod\limits_{j= i}^{m+i-1}b_{j,n+m-1} \\
1   &\cdots    & 1 & 1 & a_{m+i}&  \cdots   & b_{m+i, n+m-1}\\
\cdot  & \cdot &\cdot   & \cdot & \cdot &\cdot  & \cdot   \\
1   &\cdots    & 1 & 1 & 1 &  \cdots   &  a_{n+m-1}  
\end{array} } \right)\right), 
\end{align}
where
\begin{align*}
T^i_k:=\otimes \left( {\begin{array}{ccccc}
a_i   & b_{i,1}   & \cdots   & b_{i,k-1}   & b_{i,k}\\
1       & a_{i+1}      & \cdots   & b_{i+1, k-1}  & b_{i+1,k} \\
\cdot  & \cdot    & \cdots   & \cdot     & \cdot\\
1       & 1         & \cdots   &  a_{k-1}    & b_{k-1, k}\\
1       & 1         & \cdots   & 1           & a_k
\end{array} }\right)
\end{align*}

Moreover, consider an element $\mu\in \mathcal{O}^2=C^2((A,B,\varepsilon);A)$ given by
\begin{equation}\label{mult}
\mu\left(\otimes\left(\begin{array}{cc}
a_1 & b\\
1  & a_2
\end{array}
\right)\right)=\varepsilon(b)a_1 a_2.
\end{equation}
Then, it follows that $\{\mathcal{O}^n\}_{n\geq 0}$ is a multiplicative non-symmetric operad with the multiplication $\mu\in \mathcal{O}^2$ and $1\in \mathcal{O}^0=A$ (see \cite{Staic3} for more details).

The precise construction of Gerstenhaber algebra structure on $H^{\bullet}((A,B,\varepsilon);A)$ is described as follows: the pre-Lie bracket $\circ$ is given by
\begin{align*}
f \circ g = \sum_{i=1}^n (-1)^{(i-1)(m-1)} f \circ_i g, ~~~~ \text{ for } f \in C^n, ~ g \in C^m;
\end{align*}

\noindent the graded Lie algebra structure on $ \oplus_{n \geq 1} C^n  ((A, B , \varepsilon); A)$ is given by
\begin{align*}
[f,g] = f \circ g - (-1)^{(n-1)(m-1)} g \circ f, \quad \quad\text{ for } f \in C^n, ~ g \in C^m;
\end{align*}
the cup-product on $C^\bullet  ((A, B , \varepsilon); A)$ is defined by
\begin{align*}
(f \smile g) \left( \otimes \left( {\begin{array}{ccccc}
a_1     & b_{1,2}   & \cdots   & b_{1,n+m} \\
1       & a_2       & \cdots   & b_{2, n+m}\\
\cdot  & \cdot    & \cdot   & \cdot    \\
1       & 1         & \cdots   &  a_{m+n}  
\end{array} } \right) \right) = \quad\quad\prod_{i=1}^m ~~\prod_{j=m+1}^{m+n} \quad\quad\varepsilon (b_{i,j})   g (T^1_m) f (T^{m+1}_{m+n}). 
\end{align*}
The induced operations on the secondary Hochschild cohomology $H^\bullet ((A, B, \varepsilon); A)$ make it a Gerstenhaber algebra.

\subsection{Secondary Hochschild (co)homology associated to the triple $(A,B,\varepsilon)$}

We recall the notion of secondary Hochschild homology associated to the triple $(A,B,\varepsilon)$ from \cite{Staic1}. Let us consider the chain complex $\big(\overline{C}_{\bullet}(A,B,\varepsilon):=\oplus_{n\geq 0}\overline{C}_{n}(A,B,\varepsilon), \partial\big)$, where
 $$\overline{C}_{n}(A,B,\varepsilon):=A^{\otimes (n+1)}\otimes B^{\frac{n(n+1)}{2}}, \quad n\geq 0$$
and the differential $\partial:\overline{C}_{n}\rightarrow \overline{C}_{n-1}$ is defined by

\begin{align*}
&\partial\left(\otimes \left( {\begin{array}{ccccc}
a_0   & b_{0,1}   & \cdots   & b_{0,n-1}   & b_{0,n}\\
1       & a_1       & \cdots   & b_{1, n-1}  & b_{1,n} \\
\cdot  & \cdot    & \cdot   & \cdot      & \cdot\\
1       & 1         & \cdots   &  a_{n-1}    & b_{n-1, n}\\
1       & 1         & \cdots   & 1           & a_n
\end{array} } \right)\right)\\
:=& \sum_{i=0}^{n-1}(-1)^i~~\left(\otimes \left( {\begin{array}{ccccccc}
a_0  & \cdots & b_{0, i-1} &  b_{0, i} b_{0, i+1} & b_{0, i+2} & \cdots  & b_{0,n} \\
\cdot  & \cdot    & \cdot   & \cdot      & \cdot    &    \cdot    &  \cdot \\
1 & \cdots & a_{i-1} & b_{i-1, i} b_{i-1, i+1} & b_{i-1} b_{i+2} & \cdots & b_{i-1, n} \\
1 & \cdots & 1 & \varepsilon (b_{i, i+1}) a_i a_{i+1} & b_{i, i+2} b_{i+1, i+2} & \cdots & b_{i,n} b_{i+1,n} \\
1& \cdots & 1 & 1 & a_{i+2} & \cdots & b_{i+2, n} \\
\cdot & \cdot & \cdot & \cdot & \cdot & \cdot & \cdot \\
1 & \cdots & 1 & 1 & 1 & \cdots & a_{n}
\end{array} } \right)\right)\\
&+(-1)^n\left(\otimes \left( {\begin{array}{ccccccc}
\varepsilon(b_{0,n}) a_n a_0   & b_{1,n}b_{0,1}&\cdots & b_{i, n}b_{0, i}  &  \cdots  & b_{n-1,n}b_{0,n-1} \\
1& a_1 &\cdots & b_{1,i}& \cdots & b_{1,n-1}\\
\cdot  & \cdot    & \cdot   & \cdot   &\cdot   & \cdot \\
1 & \cdots & \cdots & a_{i}  &  \cdots & b_{i, n-1} \\
\cdot & \cdot & \cdot & \cdot  & \cdot & \cdot \\
1 & \cdots & \cdots & 1 & \cdots & a_{n-1}
\end{array} } \right)\right).
\end{align*}
The homology of the complex $\big(\overline{C}_{\bullet}(A,B,\varepsilon), \partial\big)$ is called the secondary Hochschild homology associated to the triple $(A,B,\varepsilon)$ and denoted by $HH_{\bullet}(A,B,\varepsilon)$. The homology $HH_{\bullet}(A,B,\varepsilon)$ is different from the notion of secondary Hochschild homology of the triple $(A,B,\varepsilon)$ with coefficients in the module $A$. For the definition of $HH^{\bullet}(A,B,\varepsilon)$, the secondary Hochschild cohomology associated to a triple $(A,B,\varepsilon)$, one should see the detailed construction in Section $4$ of \cite{Staic4}. Here, we skip the definition of $HH^{\bullet}(A,B,\varepsilon)$ since we do not use it in the present article.

\section{Calculus structure on secondary Hochschild (co)homology}
In this section, we define comp module action of $\mathcal{O}^{\bullet}=C^{\bullet}((A,B,\varepsilon);A)$ on the secondary Hochschild complex $\overline{C}_{\bullet}(A,B,\varepsilon)$. We prove that this comp module action makes $\overline{C}_{\bullet}(A,B,\varepsilon)$ a cyclic comp module over the operad $C^{\bullet}((A,B,\varepsilon);A)$. We conclude that $\big(H^{\bullet}((A,B,\varepsilon);A),HH_{\bullet}(A,B,\varepsilon)\big)$ forms a (noncommutative) differential calculus.

\subsection{Cyclic comp module structure on $\overline{C}_{\bullet}(A,B,\varepsilon)$}

We denote $\mathcal{M}_n:=\overline{C}_n(A,B,\varepsilon),$ for $n\geq 0$. Let us define $k$-linear maps $\bullet_i:\mathcal{O}^m\otimes \mathcal{M}_p\rightarrow \mathcal{M}_{p-m+1}$ as follows
\begin{align}\label{comp module maps}
&{f}\bullet_i\left(\otimes \left( {\begin{array}{ccccc}
a_0   & b_{0,1}   & \cdots   & b_{0,p-1}   & b_{0,p}\\
1       & a_1       & \cdots   & b_{1, p-1}  & b_{1,p} \\
\cdot  & \cdot     & \cdot    & \cdot       & \cdot \\
1        & 1         & \cdots   &  a_{p-1}    & b_{p-1, p}\\
1       & 1         & \cdots   & 1           & a_p
\end{array} } \right)\right)\\\nonumber
=&\otimes \left( {\begin{array}{ccccccccc}
a_0 &\cdots    & b_{0,i-1} & \prod\limits_{j=i}^{m+i-1}b_{0,j} & b_{0,m+i}&\cdots   & b_{0,p} \\
1   &\cdots    & b_{1,i-1} & \prod\limits_{j=i}^{m+i-1}b_{1,j} & b_{1,m+i}&  \cdots   & b_{1, p}\\
\cdot  & \cdots &\cdot   & \cdot & \cdot &\cdots  & \cdot   \\
1 &\cdots    & a_{i-1} & \prod\limits_{j=i}^{m+i-1}b_{i-1,j} & b_{i-1,m+i}&\cdots   & b_{i-1,p} \\
1 &\cdots    & 1 & {f\big(T^{i}_{m+i-1}\big)} &\prod\limits_{j=i}^{m+i-1}b_{j,m+i}&\cdots   & \prod\limits_{j=i}^{m+i-1}b_{j,p} \\
1   &\cdots    & 1 & 1 & a_{m+i}&  \cdots   & b_{m+i, p}\\
\cdot  & \cdots &\cdot   & \cdot & \cdot &\cdots  & \cdot   \\
1   &\cdots    & 1 & 1 & 1 &  \cdots   &  a_{p}  
\end{array} } \right),
\end{align}
for $f\in \mathcal{O}^m,$ $0\leq m\leq p$, and $0\leq i\leq p-m+1$. 
\medskip

Let ${f}\in \mathcal{O}^m,~{g}\in \mathcal{O}^n,$ and ${T}\in \mathcal{M}_p$. Then, for $m\geq 0,~ n,p\geq 0,$ we first show that 

\begin{equation}\label{comp module relations}
{f}\bullet_i ({g}\bullet_j T) = \threepartdef
{{g}\bullet_j ({f}\bullet_{i+n-1} T)}      {j<i,}
{({f}\circ_{j-i+1} {g})\bullet_i T}      {j-m<i\leq j,}
{{g}\bullet_{j-m+1} ({f}\bullet_i T)} {0\leq i\leq j-m.}
\end{equation} 
The left hand side is given by 
\begin{equation}
{f}\bullet_i({g}\bullet_j {T})
={f}\bullet_i\left(\otimes \left( {\begin{array}{ccccccccc}
a_0 &\cdots    & b_{0,j-1} & \prod\limits_{t=j}^{n+j-1}b_{0,t} & b_{0,n+j}&\cdots   & b_{0,p} \\
1   &\cdots    & b_{1,j-1} & \prod\limits_{t=j}^{n+j-1}b_{1,t} & b_{1,n+j}&  \cdots   & b_{1, p}\\
\cdot  & \cdots &\cdot   & \cdot & \cdot &\cdots  & \cdot   \\
1 &\cdots    & a_{j-1} & \prod\limits_{t=j}^{n+j-1}b_{j-1,t} & b_{j-1,n+j}&\cdots   & b_{j-1,p} \\
1 &\cdots    & 1 & {g\big(T^{j}_{n+j-1}\big)}&\prod\limits_{s=j}^{n+j-1}b_{s,n+j}&\cdots   & \prod\limits_{s=j}^{n+j-1}b_{s,p} \\
1   &\cdots    & 1 & 1 & a_{n+j}&  \cdots   & b_{n+j, p}\\
\cdot  & \cdots &\cdot   & \cdot & \cdot &\cdots  & \cdot   \\
1   &\cdots    & 1 & 1 & 1 &  \cdots   &  a_{p}  
\end{array} } \right)\right).
\end{equation}
\noindent \textbf{Case 1}: For $j<i$, we have the following expression 
\begin{align*}
&{f}\bullet_i({g}\bullet_j {T})\\
=&{f}\bullet_i\left(\otimes \left( {\begin{array}{ccccccccc}
a_0 &\cdots    & b_{0,j-1} & \prod\limits_{t=j}^{n+j-1}b_{0,t} &\cdots& \prod\limits_{t=i+n-1}^{i+m+n-2}b_{0,t}&\cdots   & b_{0,p} \\
\cdot  & \cdots &\cdot   & \cdot & \cdots &\cdot  & \cdots&\cdot   \\
1 &\cdots    & a_{j-1} & \prod\limits_{t=j}^{n+j-1}b_{j-1,t} &\cdots& \prod\limits_{t=i+n-1}^{i+m+n-2}b_{j-1,t}&\cdots   & b_{j-1,p} \\
1 &\cdots    & 1 & {g\big(T^{j}_{n+j-1}\big)}&\cdots&\prod\limits_{s=j}^{n+j-1}\prod\limits_{t=i+n-1}^{i+m+n-2}b_{s,t}&\cdots   & \prod\limits_{s=j}^{n+j-1}b_{s,p} \\
\cdot  & \cdots &\cdot   & \cdot & \cdots &\cdot  & \cdots&\cdot   \\
1 &\cdots    & 1 & 1& \cdots &{f\big(T^{i+n-1}_{i+m+n-2}\big)}&\cdots   & \prod\limits_{s=i+n-1}^{i+m+n-2}b_{s,p} \\
\cdot  & \cdots &\cdot   & \cdot & \cdots &\cdot  & \cdots&\cdot   \\
1   &\cdots    & 1 & 1 & \cdots & 1& \cdots   &  a_{p}  
\end{array} } \right)\right)\\
=&{g}\bullet_j\left(\otimes \left( {\begin{array}{ccccccccc}
a_0 &\cdots    & b_{0,j-1} & \prod\limits_{t=i+n-1}^{i+m+n-2}b_{0,t}&\cdots   & b_{0,p} \\
\cdot  & \cdots &\cdot   & \cdot & \cdots &\cdot  \\
1 &\cdots    & a_{i+n-2} &  \prod\limits_{t=i+n-1}^{i+m+n-2}b_{j-1,t}&\cdots   & b_{j-1,p} \\
1 &\cdots    & 1 &  {f\big(T^{i+n-1}_{i+m+n-2}\big)}&\cdots   & \prod\limits_{s=i+n-1}^{i+m+n-2}b_{s,p} \\
\cdot  & \cdots &\cdot   & \cdot & \cdots &\cdot  \\
1   &\cdots    & 1 & 1 & \cdots &  a_{p}  
\end{array} } \right)\right)\\
=&{g\bullet_j\Bigg(f\bullet_{i+n-1}}\left(\otimes \left( {\begin{array}{ccccc}
a_0   & b_{0,1}   & \cdots   & b_{0,p-1}   & b_{0,p}\\
1       & a_1       & \cdots   & b_{1, p-1}  & b_{1,p} \\
\cdot  & \cdot     & \cdot    & \cdot       & \cdot \\
1        & 1         & \cdots   &  a_{p-1}    & b_{n-1, p}\\
1       & 1         & \cdots   & 1           & a_p
\end{array} } \right)\right)\Bigg)\\
=&{g}\bullet_j({f}\bullet_{i+n-1}T).
\end{align*}

\medskip

\noindent By a similar calculation, the following cases hold.

\medskip

\noindent \textbf{Case 2:} For $j-m<i\leq j$, we get ${f}\bullet_i({g}\bullet_j {T})={({f}\circ_{j-i+1} {g})\bullet_{i}T}.$

\medskip

\noindent \textbf{Case 3}: For $0\leq i\leq j-m$, we get ${f}\bullet_i({g}\bullet_j {T})={g\bullet_{j-m+1}(f}\bullet_{i}T).$

\medskip
\noindent \textbf{Unitality Condition:} For $\textbf{1}=Id_A\in \mathcal{O}^1=C^1((A,B,\varepsilon);A)$, it is clear that 
$$\textbf{1}\bullet_i T=T, \quad\quad\mbox{for}~~T\in \mathcal{M}_n,~~\mbox{and~~} 0\leq i\leq n.$$ 

We define a map $\mathfrak{t}:\mathcal{M}_p\rightarrow \mathcal{M}_p$ as follows
\begin{equation}\label{cyclic operator def}
\mathfrak{t}\left(\otimes \left( {\begin{array}{ccccc}
a_0   & b_{0,1}   & \cdots   & b_{0,p-1}   & b_{0,p}\\
1       & a_1       & \cdots   & b_{1, p-1}  & b_{1,p} \\
\cdot  & \cdot     & \cdot    & \cdot       & \cdot \\
1        & 1         & \cdots   &  a_{p-1}    & b_{p-1, p}\\
1       & 1         & \cdots   & 1           & a_p
\end{array} } \right)\right)=\otimes \left( {\begin{array}{ccccc} 
a_p & b_{0,p} & b_{1,p} & \cdots& b_{p-1,p}\\
1& a_0   & b_{0,1}   & \cdots  &  b_{0,p-1}\\
1&1       & a_1       & \cdots   & b_{1,p-1} \\
\cdot&\cdot  & \cdot     & \cdot        & \cdot \\
1&1        & 1         & \cdots      & a_{p-1}
\end{array} } \right).
\end{equation}

Next, we show that the map $\mathfrak{t}:\mathcal{M}_p\rightarrow \mathcal{M}_p$ makes the unital comp module $\mathcal{M}_{\bullet}$, a cyclic comp module over the opeard $\mathcal{O}^{\bullet}$. 

For $T\in \mathcal{M}_p$ and ${f}\in \mathcal{O}^{\bullet}$, we have the following expression
 \begin{align*}
\mathfrak{t}~{\big({f}\bullet_i T\big)}
=&\mathfrak{t}\left(\otimes \left( {\begin{array}{ccccccccc}
a_0 &\cdots    & b_{0,i-1} & \prod\limits_{t=i}^{i+m-1}b_{0,t}&\cdots   & b_{0,p} \\
\cdot  & \cdots &\cdot   & \cdot & \cdots &\cdot \\
1 &\cdots    & a_{i-1} &  \prod\limits_{t=i}^{i+m-1}b_{i-1,t}&\cdots   & b_{i-1,p} \\
1 &\cdots    & 1 &  { f\big(T^{i}_{i+m-1}\big)}&\cdots   & \prod\limits_{s=i}^{m+i-1}b_{s,p} \\
\cdot  & \cdots &\cdot   & \cdot & \cdots &\cdot   \\
1   &\cdots    & 1 & 1 & \cdots &  a_{p}  
\end{array} } \right)\right)\\
=&\otimes \left( {\begin{array}{ccccccccc}
a_{p} & b_{0,p}&\cdots &b_{i-1,p}&\prod\limits_{s=i}^{m+i-1}b_{s,p}&\cdots& b_{p-1,p}\\
1&a_0 &\cdots    & b_{0,i-1} & \prod\limits_{t=i}^{i+m-1}b_{0,t}&\cdots   & b_{0,p-1} \\
\cdot  & \cdot &\cdots   & \cdot & \cdot &\cdots &\cdot   \\
1&1 &\cdots    & a_{i-1} &  \prod\limits_{t=i}^{i+m-1}b_{i-1,t}&\cdots   & b_{i-1,p-1} \\
1 &1&\cdots    & 1 &  {f\big(T^{i}_{i+m-1}\big)}&\cdots   & \prod\limits_{s=i}^{m+i-1}b_{s,p-1} \\
\cdot  & \cdot &\cdots   & \cdot & \cdot &\cdots&\cdot   \\
1   &1&\cdots    & 1 & 1 & \cdots &  a_{p-1}  
\end{array} } \right)\\
=&{f}\bullet_{i+1}\left(\otimes \left( {\begin{array}{ccccc} 
a_p & b_{0,p} & b_{1,p} & \cdots& b_{p-1,p}\\
1& a_0   & b_{0,1}   & \cdots  &  b_{0,p-1}\\
1&1       & a_1       & \cdots   & b_{1,p-1} \\
\cdot&\cdot  & \cdot     & \cdot        & \cdot \\
1&1        & 1         & \cdots      & a_{p-1}
\end{array} } \right)\right) ={f}\bullet_{i+1}\mathfrak{t}{(T)}.
\end{align*}
\noindent For $i\geq 1$,
$$\mathfrak{t}^i(T)=
\otimes \left( {\begin{array}{ccccccccc} 
a_{p-i+1} & b_{p-i+1,p-i+2}&\cdots & b_{p-i+1,p}& b_{0,p-i+1} & b_{1,p-i+1} & \cdots & b_{p-i,p-i+1}\\
1&a_{p-i+2} & \cdots & b_{p-i+2,p}& b_{0,p-i+2} & b_{1,p-i+2} & \cdots & b_{p-i,p-i+2}\\
\cdot & \cdot & \cdot & \cdot & \cdot & \cdot & \cdot & \cdot\\
1 & 1 &\cdots & a_{p}& b_{0,p} & b_{1,p} & \cdots& b_{p-i,p}\\
1& 1& \cdots & 1& a_0   & b_{0,1}   & \cdots  &  b_{0,p-i}\\
1&1& \cdots & 1& 1       & a_1       & \cdots   & b_{1,p-i}\\
\cdot & \cdot & \cdot & \cdot & \cdot & \cdot & \cdot & \cdot\\
1&1& \cdots & 1& 1 & 1 & \cdots & a_{p-i}
\end{array} } \right)$$

\noindent Clearly, it follows that the map $\mathfrak{t}:\mathcal{M}_p\rightarrow \mathcal{M}_p$ satisfies the identity $\mathfrak{t}^{p+1}=Id$. Thus,
$$\mathfrak{t}^{p+1}=Id\quad \mbox{and}\quad \mathfrak{t} ~{\big({f}\bullet_i T\big)}={f}\bullet_{i+1} \mathfrak{t}~{(T)}.$$
Therefore, by the above discussion we have the following theorem.
\begin{thm}
Let $(A,B,\varepsilon)$ be a triple and $\mathcal{O}^{\bullet}:=\{\mathcal{O}^n\}_{n\geq 0}$ be the operad structure on the complex $C^{\bullet}((A,B,\varepsilon);A)$. Then the complex $\overline{C}_{\bullet}(A,B,\varepsilon)$ associated to the triple $(A,B,\varepsilon)$ is a cyclic comp module over the operad $\mathcal{O}^{\bullet}$.
\end{thm}
\qed
\subsection{Cyclic $k$-module structure on $\overline{C}_{\bullet}(A,B,\varepsilon)$}
The complex $\mathcal{M}_{\bullet}=\overline{C}_{\bullet}(A,B,\varepsilon)$ is a cyclic comp module over the multiplicative operad $(\mathcal{O}^{\bullet}=C^{\bullet}((A,B,\varepsilon);A),\mu)$. It yields the following cyclic $k$-module structure on $\overline{C}_{\bullet}(A,B,\varepsilon)$: 
 
\begin{enumerate}[(i)]
\item for $i=0,1,\ldots,p,$ the face maps $d_i:\mathcal{M}_p\rightarrow \mathcal{M}_{p-1}$ are given as follows

\begin{enumerate}[(a)]
\item if $0\leq i\leq p-1$,
$$d_i(T)=\left(\otimes \left( {\begin{array}{cccccc}
a_0 & \cdots & b_{0,i-1}& b_{0,i}b_{0,i+1}  & \cdots   & b_{0,p}\\
\cdot  & \cdot    & \cdot  &\cdot &\cdot  & \cdot \\
1       & \cdots & a_{i-1} & b_{i-1,i}b_{i-1,i+1} & \cdots    & b_{i-1,p} \\
1&\cdots &1 & a_i a_{i+1}\varepsilon(b_{i,i+1})& \cdots & b_{i,n}b_{i+1,n}\\
\cdot  & \cdot     & \cdot  &\cdot &\cdot  & \cdot \\
1     &\cdots  & 1      &1   & \cdots    & a_p
\end{array} } \right)\right),$$

\item if $i=p$,
$$d_p(T)=\left(\otimes \left( {\begin{array}{ccccc}
\varepsilon(b_{0,n})a_n a_0  &b_{1,p}b_{0,1}  & \cdots   & b_{p-1,p} b_{0,p-1}\\
1       & a_1       & \cdots    & b_{1,p-1} \\
\cdot  & \cdot     & \cdot    & \cdot \\
1        & 1         & \cdots  & b_{p-2, p-1}\\
1       & 1         & \cdots    & a_{p-1}
\end{array} } \right)\right);$$

\end{enumerate} 

\medskip

\item for $j=0,1,\cdots,p,$ the degeneracies $s_j:\mathcal{M}_p\rightarrow \mathcal{M}_{p+1}$ are given by
$$s_j \left(\otimes \left( {\begin{array}{ccccc}
a_0   & b_{0,1}   & \cdots    & b_{0,p}\\
1       & a_1       & \cdots & b_{1,p} \\
\cdot  & \cdot     & \cdot          & \cdot \\
1        & 1         & \cdots    & b_{p-1, p}\\
1       & 1         & \cdots   & a_p
\end{array} } \right)\right)=\otimes \left( {\begin{array}{cccc|c|ccc} 
a_0   & b_{0,1}   & \cdots   & b_{0,j}  &1 &1&\cdots & b_{0,p}\\
1       & a_1       & \cdots   & b_{1, j-1}  &1 &1&\cdots & b_{1,p} \\
\cdot  & \cdot     & \cdot    & \cdot &\cdot &\cdot &\cdot      & \cdot \\
1       & 1& \cdots & a_{j}   & 1   & 1 &\cdots & b_{j,p}\\
\hline
1        & 1         & \cdots   & 1& 1_A &1 &\cdots    & 1\\
\hline
1       & 1& \cdots & 1   & 1   & a_{j+1} &\cdots & b_{j+1,p}\\
\cdot  & \cdot     & \cdot    & \cdot &\cdot &\cdot &\cdot      & \cdot \\
1       & 1         & \cdots   & 1 &1&1 &\cdots          & a_p
\end{array} } \right);$$
\medskip
\item the cyclic operator $\mathfrak{t}:\mathcal{M}_p\rightarrow \mathcal{M}_{p}$ is given by the equation \eqref{cyclic operator def}.
\end{enumerate}

 \medskip

The simplicial boundary operator on the complex $\mathcal{M}_{\bullet}$ is given by 
$$\mathfrak{b}:=\sum_{i=0}^p (-1)^i d_i.$$
Note that the complex $(\mathcal{M}_{\bullet},\mathfrak{b})$ is the same as the secondary Hochschild chain complex associated to the triple $(A,B,\varepsilon)$, defined in Section $3$. Thus, the associated homology $$H_{\bullet}(\mathcal{M})=HH_{\bullet}(A,B,\varepsilon).$$ 

Next, let us recall that the extra degeneracy map, the norm operator and the cyclic differential are given by  
$$s_{-1}(T):=\mathfrak{t}~s_p(T)=1_A\bullet_0 T,\quad N:=\sum_{p=0}^n (-1)^{ip}\mathfrak{t}^i, ~\quad \mbox{and}\quad B:=(Id-t)s_{-1}N.$$
 The normalised complex $\mathfrak{N}(\mathcal{M})_{\bullet}$ is the quotient of the complex $\mathcal{M}_{\bullet}$ by the subcomplex spanned by $\{ 1_A\bullet_{j+1}-, j=0,1,\cdots,p\}$. The cyclic differential $B$ induces the following map on the normalised complex:
\begin{equation}\label{Conne's operator}
B(T)=s_{-1}N(T)=\sum_{i=0}^p (-1)^{ip}~ 1_A\bullet_0 \mathfrak{t}^i (T),\quad \mbox{for  }~~T\in \mathfrak{N}(\mathcal{M})_{p}.
\end{equation}
 
\subsection{Cosimplicial $k$-module structure on the   complex $C^{\bullet}((A,B,\varepsilon);A)$} Now, we consider the cosimplicial $k$-module structure on $\mathcal{O}^{\bullet}$ associated to the multiplicative operad $(\mathcal{O}^{\bullet},\mu)$. Then the face maps $d^i:\mathcal{O}^p\rightarrow \mathcal{O}^{p+1}$ and $s^j:\mathcal{O}^p\rightarrow \mathcal{O}^{p-1}$ are given as follows
 \begin{equation*}
d^i({f}) = \threepartdef
{\mu \circ_2 {f}}      {i=0,}
{{f}\circ_{i} \mu}      {1\leq i\leq p,}
{\mu\circ_1{f}} {i=p+1}
\end{equation*} 
and
 \begin{equation*}
s^j({f}) = {f}\circ_{i+1} 1_A, \quad \mbox{for ~~}j=0,1,\cdots,p-1.
\end{equation*} 
In particular, 

$$s^j(f) \left(\otimes \left( {\begin{array}{ccccc}
a_1   & b_{1,2}   & \cdots    & b_{1,p-1}\\
1       & a_2       & \cdots & b_{2,p-1} \\
\cdot  & \cdot     & \cdot          & \cdot \\
1        & 1         & \cdots    & b_{p-2, p-1}\\
1       & 1         & \cdots   & a_{p-1}
\end{array} } \right)\right)={f}\left(\otimes \left( {\begin{array}{cccc|c|ccc} 
a_1   & \cdots   & b_{1,j}  &1 &1&\cdots & b_{1,p-1}\\
1      & \cdots   & b_{1, j-1}  &1 &1&\cdots & b_{2,p-1} \\
\cdot     & \cdot    & \cdot &\cdot &\cdot &\cdot      & \cdot \\
1       & \cdots & a_{j-1}   & 1   & 1 &\cdots & b_{j-1,p-1}\\
\hline
1              & \cdots   & 1& 1_A &1 &\cdots    & 1\\
\hline
1       & \cdots & 1   & 1   & a_{j} &\cdots & b_{j,p-1}\\
\cdot  & \cdot    & \cdot &\cdot &\cdot &\cdot      & \cdot \\
1    & \cdots   & 1 &1&1 &\cdots          & a_{p-1}
\end{array} } \right)\right).$$

\medskip

The coboundary map $\delta^p:\mathcal{O}^p\rightarrow \mathcal{O}^{p+1}$ is given by $\delta^p:=\sum_{i=0}^{p+1}d^i.$ The cochain complex $(\mathcal{O}^{\bullet},\delta)$ coincides with the complex $(C^{\bullet}((A,B,\varepsilon);A),\delta^{\varepsilon
})$. The conormalised cochain complex is given by
$$\mathfrak{N}(\mathcal{O})^p:=\cap_{j=0}^{p-1}Ker(s^j).$$
Here, $\mathfrak{N}(\mathcal{O})^p$ is the collection of maps in $Hom_k (A^{\otimes p} \otimes B^{\otimes \frac{p(p-1)}{2}} , A)$, which vanish at the elements of the type
 $$\left(\otimes \left( {\begin{array}{ccc|c|cccc} 
a_1   & \cdots   & b_{1,j}  &1 &1&\cdots & b_{1,p-1}\\
1      & \cdots   & b_{1, j-1}  &1 &1&\cdots & b_{2,p-1} \\
\cdot     & \cdot    & \cdot &\cdot &\cdot &\cdot      & \cdot \\
1       & \cdots & a_{j-1}   & 1   & 1 &\cdots & b_{j-1,p-1}\\
\hline
1              & \cdots   & 1& 1_A &1 &\cdots    & 1\\
\hline
1       & \cdots & 1   & 1   & a_{j} &\cdots & b_{j,p-1}\\
\cdot  & \cdot    & \cdot &\cdot &\cdot &\cdot      & \cdot \\
1    & \cdots   & 1 &1&1 &\cdots          & a_{p-1}
\end{array} } \right)\right).$$

\medskip

\noindent Note that the cohomology $H^{\bullet}(\mathfrak{N}(\mathcal{O})^{\bullet})=H^{\bullet}_{\mu}({\mathcal{O}^{\bullet}})=H^{\bullet}((A,B,\varepsilon);A)$.

\subsection{A calculus structure on $\big(H^{\bullet}((A,B,\varepsilon);A),HH_{\bullet}(A,B,\varepsilon)\big)$}
 
The complex $\mathcal{M}_{\bullet}:=\overline{C}_{\bullet}(A,B,\varepsilon)$ is a cyclic comp module over the multiplicative operad $\mathcal{O}=C^{\bullet}((A,B,\varepsilon);A)$. Let us recall from Section 2 that a cyclic (unital) comp module $\mathcal{M}_{\bullet}$ over a multiplicative operad $(\mathcal{O},\mu)$ induces a calculus structure on the pair $\big(H_{\mu}^{\bullet}(\mathcal{O}),H_{\bullet}(\mathcal{M})\big)$. From Subsections 4.2 and 4.3, it follows that 
\begin{align*}
H_{\bullet}(\mathcal{M})&=HH_{\bullet}(A,B,\varepsilon),\\
H^{\bullet}_{\mu}({\mathcal{O}^{\bullet}})&=H^{\bullet}((A,B,\varepsilon);A).
\end{align*}
Therefore, we obtain the following result.

\begin{thm}\label{diffcal}
The pair $\big(H^{\bullet}((A,B,\varepsilon);A),HH_{\bullet}(A,B,\varepsilon)\big)$ forms a noncommutative differential calculus.
\end{thm}
\qed
\newpage
\section{conclusion}
M. Staic first introduced secondary Hochschild (co)homology \cite{Staic1,Staic2} in order to study the deformation theory of associative algebras over a commutative ring. The secondary Hochschild (co)homology behaves similar to the Hochschild (co)homology in several aspects. However, unlike Hochschild (co)homology, there is no natural cyclic action on the (co)chain complex. The absence of a natural cyclic action also explains the natural constructions of new complexes while introducing secondary cyclic (co)homology \cite{Staic4}. 

In the Hochschild case, there is a noncommutative differential calculus $(HH^\bullet(A,A),HH_{\bullet}(A,A))$. Under certain conditions \cite{Ginzburg}, the isomorphism $HH^\bullet(A,A)\cong HH_{\bullet}(A,A)$ yields a BV algebra structure on the cohomology $HH^\bullet(A,A)$. In fact, the condition obtained by \cite{Ginzburg} (in the case of Calabi-Yau algebras) can be written in general for any calculus: if $(\mathcal{A},\Omega)$ is a calculus, then 
\begin{align}\label{BV-relation}
\nonumber
i_{[f,g]}&=[i_f,\mathcal{L}_g]\\\nonumber
&=i_f\circ \mathcal{L}_g-(-1)^{m(n+1)}\mathcal{L}_g\circ i_f\\\nonumber
&=i_f\circ (B\circ i_g-(-1)^n i_g\circ B)-(-1)^{m(n+1)}(B\circ i_g-(-1)^n i_g\circ B)\circ i_f\\
&=i_f\circ B\circ i_g -(-1)^{n} i_{f\smile g}\circ B -(-1)^{m(n+1)}(-1)^{mn} B\circ i_{f\smile g}+(-1)^{m(n+1)+n} i_g\circ B\circ i_f
\end{align}
for $f\in \mathcal{A}^m,~g\in \mathcal{A}^n$. Subsequently, we obtain the following result which gives a condition on the calculus such that the underlying Gerstenhaber algebra is a Batalin-Vilkovisky (BV) algebra.

\begin{thm}\label{diffcal-BV}
Let $(\mathcal{A},\Omega)$ be a calculus. If there exists an element $c\in \Omega_k$ such that $B(c)=0$ and the map $\Theta:\mathcal{A}^\bullet\rightarrow \Omega_{k-\bullet}$, given by $\Theta(f)=i_fc$ is an isomorphism, then the map $\Delta:\mathcal{A}^\bullet\rightarrow \mathcal{A}^{\bullet-1},$ defined by $i_{\Delta f}c=B(i_f c)$
makes the Gerstenhaber algebra $\mathcal{A}$ into a BV algebra.   
\end{thm}

\begin{proof}
Since $B(c)=0$ and $i_{\Delta f}c=B(i_f c)$, by equation \eqref{BV-relation} it follows that 
$$i_{[f,g]}(c)=-(-1)^{m}\big(i_{\Delta(f\smile g)}c -i_{\Delta(f)\smile g}c -(-1)^m i_{f\smile\Delta(g)}c\big).$$
In turn the condition that $\Theta$ is an isomorphism implies that the operator $\Delta$ generates the Gerstenhaber bracket on $\mathcal{A}$, i.e. 
$$[f,g]=-(-1)^{m}\big(\Delta(f\smile g)- \Delta(f)\smile g  -(-1)^m  f\smile\Delta(g)\big),\quad\mbox{for all } f\in \mathcal{A}^m,~g\in \mathcal{A}^n.$$
\end{proof}

The above theorem was first proved by T. Lambre in \cite{Lambre}. Theorem \ref{diffcal-BV} gives a condition on the calculus $\big(H^{\bullet}((A,B,\varepsilon);A),HH_{\bullet}(A,B,\varepsilon)\big)$ such that the Gerstenhaber algebra structure on $H^{\bullet}((A,B,\varepsilon),A)$ becomes a BV algebra structure. More precisely, if we have an element $[T]\in HH_{\bullet}(A,B,\varepsilon)$ such that $B[T]=0$ and the map $$\Theta: H^{\bullet}((A,B,\varepsilon),A)\rightarrow HH_{k-\bullet}(A,B,\varepsilon),~~\mbox{ given by }\Theta(f)=i_f[T]$$ is an isomorphism, then $H^{\bullet}((A,B,\varepsilon),A)$ is a BV algebra. 

In the Hochschild case, if $A$ is symmetric algebra then the calculus $(HH^\bullet(A,A),HH_{\bullet}(A,A))$ satisfies the conditions in the Theorem \ref{diffcal-BV} and the Hochschild cohomology $HH^\bullet(A,A)$ carries a BV algebra structure (see \cite{Menichi04, Menichi, Tradler} for details). It will be interesting to find conditions on the $B$-algebra $A$ such that the conditions in Theorem \ref{diffcal-BV} hold. It needs further investigation since 1) there is no derived functor description for the secondary Hochschild (co)homology, and 2)- even in the finite-dimensional symmetric algebra case, the most canonical choice for the BV-operator does not lift to the secondary Hochschild cohomology  \cite{BV-2020}. More precisely, in \cite{BV-2020} the authors consider a finite-dimensional symmetric algebra $A$ and define a BV-operator $\Delta$ on the homotopy Gerstenhaber algebra $C^*((A, B, \varepsilon), A)$, which due to proposition\,9 and theorem\,10 in \cite{BV-2020}, is the most canonical choice for a square zero BV-differential operator on $H^*((A, B, \varepsilon), A)$. Though $\Delta$ determines the graded Lie bracket on $H^*((A, B, \varepsilon), A)$, it is not a cochain map in general (see theorem\,7, \cite{BV-2020}).

\noindent {\em Acknowledgements.} The research of S. K. Mishra is supported by the NBHM postdoctoral fellowship. The author thanks NBHM for its support.


\end{document}